\documentclass[11pt, oneside]{amsart}
 \usepackage[text={5.58in,8.5in},centering,letterpaper,dvips]{geometry}
\usepackage{graphicx}
\usepackage{amsfonts}
\usepackage{epsf}
\usepackage{amssymb}
\usepackage{amsmath}
\usepackage{amscd}
\usepackage{tikz}
\usepackage{pdfpages}
\usepackage{fancyhdr}
\usepackage{setspace}
\usepackage{hyperref}
\usepackage[all]{xy}
\usetikzlibrary{matrix}
\usepackage{verbatim}
\usepackage{enumerate}
\usepackage{amsbsy}
\usepackage{caption}
\usepackage{subcaption}

\theoremstyle{theorem}
\newtheorem{theorem}{Theorem}[section]
\newtheorem{proposition}[theorem]{Proposition}
\newtheorem{lemma}[theorem]{Lemma}
\newtheorem{question}[theorem]{Question}
\newtheorem{corollary}[theorem]{Corollary}

\theoremstyle{definition}

\newcommand{\Z}{\mathbb{Z}}

\newcommand{\Q}{\mathbb{Q}}

\newcommand{\bdy}{\partial}

\makeatletter
\def\@seccntformat#1{%
  \protect\textup{\protect\@secnumfont
    \ifnum\pdfstrcmp{subsection}{#1}=0 \bfseries\fi% subsection # in \bfseries
    \csname the#1\endcsname
    \protect\@secnumpunct
  }%
}  
\makeatother

\makeatletter
\newtheorem*{rep@theorem}{\rep@title}
\newcommand{\newreptheorem}[2]{%
\newenvironment{rep#1}[1]{%
 \def\rep@title{#2 \ref{##1}}%
 \begin{rep@theorem}}%
 {\end{rep@theorem}}}
\makeatother

\newreptheorem{theorem}{Theorem}
\newreptheorem{lemma}{Lemma}
\newreptheorem{question}{Question}
\newreptheorem{corollary}{Corollary}
\newreptheorem{proposition}{Proposition}
\newreptheorem{problem}{Problem}

\begin{document}

\rhead{\thepage}
\lhead{\author}
\thispagestyle{empty}

%\tableofcontents
%\listoffigures

\raggedbottom
\pagenumbering{arabic}
\setcounter{section}{0}

%%%%%%%%%%%%%%%%%%%%%%%%%%%%%%%%%%%%%%%%%%%%%%%%%%%%%%%%
%%%%%%%%%%%%%%%%%%%%%%%%%%%%%%%%%%%%%%%%%%%%%%%%%%%%%%%%
%%%%%%%%%%%%%%%%%%%%%%%%%%%%%%%%%%%%%%%%%%%%%%%%%%%%%%%%

\title{ An infinite family of counterexamples to Batson's conjecture}
\date{\today}

\author{Vincent Longo}
\address{Department of Mathematics, University of Nebraska-Lincoln, Lincoln, NE 68588}
\email{vincent.longo@huskers.unl.edu}
\urladdr{https://www.math.unl.edu/~vlongo2/} 
%\author{Alex Zupan}
%\address{Department of Mathematics, University of Nebraska-Lincoln, Lincoln, NE 68588}
%\email{zupan@unl.edu}
%\urladdr{http://www.math.unl.edu/~azupan2}

%
%\author{Alexander Zupan}
%\address{Department of Mathematics, University of Nebraska-Lincoln, Lincoln, NE 68588}
%\email{zupan@unl.edu}
%\urladdr{http://www.math.unl.edu/~azupan2}
%
%\author{Mark Brittenham}
%\address{Department of Mathematics, University of Nebraska-Lincoln, Lincoln, NE 68588}
%\email{mbrittenham2@unl.edu}
%\urladdr{https://www.math.unl.edu/~mbrittenham2/}

\begin{abstract}
Batson's conjecture is a non-orientable version of Milnor's conjecture, which states that the 4-ball genus of a torus knot $T(p,q)$ is equal to $\frac{(p-1)(q-1)}{2}$. Batson's conjecture states that the nonorientable 4-ball genus is equal to the pinch number of a torus knot, i.e. the number of a specific type of (nonorientable) band surgeries needed to obtain the unknot. The conjecture was recently proved to be false by Lobb. We will show that Lobb's counterexample fits into an infinite family of counterexamples. \end{abstract}

\maketitle
\section{Introduction}\label{sec:outline}
Batson's conjecture is a nonorientable analogue of Milnor's conjecture about the (orientable) 4-ball genus of torus knots. The conjecture was recently proved to be false by Lobb in \cite{Lobb}. In this paper, we provide a new infinite family of counterexamples of the form $T(4n, (2n\pm1)^2)$. See Figure \ref{T825} for one such counterexample. 

Milnor's conjecture states that the unknotting number of the $(p,q)$ torus knot $T(p,q)$ is equal to $\frac{(p-1)(q-1)}{2}$. However, since the 4-ball genus $g_4(K)$ of a knot $K$ is a lower bound for the unknotting number of $K$, it suffices to show that $g_4(T(p,q))=\frac{(p-1)(q-1)}{2}$. This was eventually verified in \cite{Kronheimer1} and \cite{Kronheimer2} using powerful tools from gauge theory. It is natural to ask if there is a similar formulation for the nonorientable 4-ball genus $\gamma_4$ for torus knots. Indeed, in \cite{Batson}, it is shown that $\gamma_4(T(2k,2k-1))$ is equal to the \emph{pinch number} of $T(2k,2k-1)$, and conjectured that $\gamma_4(T(p,q))$ is always equal to the {pinch number} of $T(p,q)$.

    \begin{figure}
\includegraphics[width=1\textwidth]{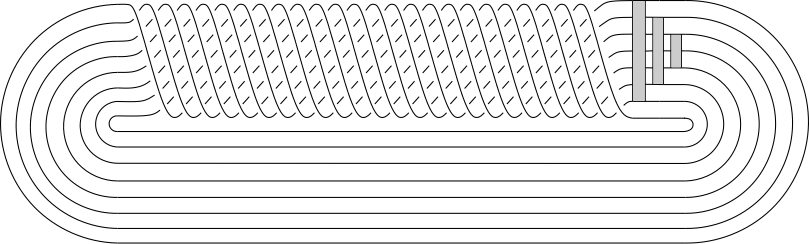}
  \caption{The knot $K_2=T(8,25)$ along with three orientation reversing bands. Surgery on these bands yields the slice knot $10_3$. Labeling the strands 1-8 from top to bottom, the bands attach strands 1 to 7, 2 to 6, and 3 to 5.}
  \label{T825}
\end{figure}

Batson's conjecture can be thought of in the following way: a pinch move on a torus knot is done by performing surgery on a (blackboard framed) band between two adjacent strands in a standard diagram of a torus knot (see Figure \ref{pinch}), called a pinch band. Note that a pinch band is necessarily nonorientable. Since the band embeds on the same torus as the torus knot that the band is attached to, then after performing the band surgery, the resulting knot still lives on a torus (and in fact is a less complicated torus knot). 
Thus, we can always find a sequence of pinch moves resulting in the unknot (called a \emph{pinch sequence}). Capping off the trace of this sequence of pinch moves yields a nonorientable surface bounded by the original torus knot. Batson's conjecture essentially asks whether there is any shorter sequence of non-orientable band surgeries taking a torus knot to the unknot (or more generally, a slice knot).

    \begin{figure}
\includegraphics[width=.7\textwidth]{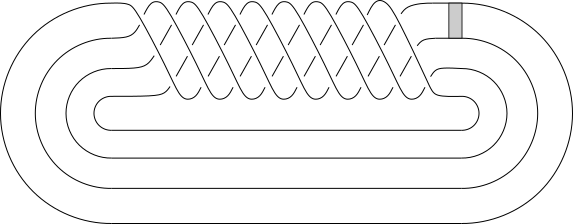}
  \caption{The knot $T(4,9)$ along with a pinch band. Since the band lies on the same torus that $T(4,9)$ lies on, surgery on this band yields another torus knot. Lobb showed that there is a single band surgery (not pictured) on $T(4,9)$ which yields Stevedore's knot, which is slice.}
  \label{pinch}
\end{figure}

Lobb disproved Batson's conjecture in \cite{Lobb} by showing that $T(4,9)$ has pinch number two and nonorientable 4-ball genus one. More specifically, Lobb found a nonorientable band surgery which takes $T(4,9)$ to Stevedore's knot, which is slice (whereas two pinch moves are needed to obtain the unknot). It was noted in \cite{Jabuka} that from this (or any other) counterexample, one can easily work backwards to find an infinite sequence of torus knots $\ldots \to T(p_n,q_n) \to T(p_{n-1},q_{n-1}) \to \ldots \to T(p_0,q_0)=T(4,9)$ where $T(p_i,q_i)\to T(p_{i-1},q_{i-1})$ denotes a pinch move taking $T(p_i,q_i)$ to $T(p_{i-1},q_{i-1})$. Then each $T(p_i,q_i)$ has pinch number $i+2$ and nonorientable 4-ball genus at most $i+1$. It is reasonable to wonder which other counterexamples exist which cannot be obtained in this way. In this paper, we provide a partial answer to this question by providing two new infinite families $\{K_n\}$ and $\{J_n\}$ of counterexamples to Batson's conjecture. 
%
%\newtheorem*{thm:batson}{Theorem \ref{thm:batson}}
%\begin{thm:batson}
%The torus knots $K_n=T(4n, (2n+1)^2)$ and $J_n=T(4n,(2n-1)^2)$ have pinch number $n$ and nonorientable 4-ball genus $n-1$. 
%
%\end{thm:batson}

\newtheorem*{thm:main}{Theorem \ref{thm:main}}
\begin{thm:main}
The torus knots $K_n=T(4n,(2n+1)^2)$ (for $n\geq 1$) and $J_n=T(4n, (2n-1)^2)$ (for $n\geq 2$) have pinch number $2n$ and nonorientable 4-ball genus at most $2n-1$. Specifically, there are $2n-1$ band surgeries for $K_n$ and $J_n$ (which are not all orientation preserving) that yield the slice knot with continued fraction expansion $[-(2n+2),-2n]$ and $[-(2n-2),-2n]$ respectively. 
\end{thm:main}

We note that Lobb's example $T(4,9)$ appears in these families as $K_1$. We also note that the knot $J_3=T(12,25)$ has a pinch sequence which passes through $T(4,9)$. In fact, we will show that for each $n>2$, four pinch moves applied to $J_n$ yields the knot $K_{n-2}$ (Corollary \ref{cor}). We note that $J_2=T(8,9)$ does not have a pinch sequence through any nontrivial $K_n$. We also show that no pinch sequence starting at any $K_n$ passes through any other $K_m$ (Corollary \ref{cor2}). The knot $K_2=T(8,25)$, which motivated these families of counterexamples, is shown in Figure \ref{T825}, along with three nonorientable band surgeries which yield the slice knot $10_3$ in the knot table.

  \section*{Acknowledgements}
  The author would like to thank his advisors Alex Zupan and Mark Brittenham for all their support, guidance, help, and suggestions throughout the process of this paper. The author would also like to thank Stanislav Jabuka and Cornelia A. Van Cott for the helpful email correspondence.

 \section{Preliminaries}

%Define $g_4(K)$, $\gamma_4(K)$, pinch moves, the pinch number, the surface carved out by a sequence of band surgeries, formally state the conjecture and perhaps the previous family of counterexamples. Should probably mention the correspondence between 2 bridge knots and rationals, as well as unions of rational tangles and $SL(2,\Z)$. 

For a knot $K\subset S^3=\bdy B^4$, we define the 4-ball genus $g_4(K)$ to be the minimum genus of an orientable surface properly embedded in $B^4$ bounded by $K$. Similarly, we define the nonorientable 4-ball genus $\gamma_4(K)$ to be the minimum value of $b_1(\Sigma)$ for a nonorientable surface $\Sigma\subset B^4$ bounded by $K$, where $b_1(\Sigma)$ denotes the first Betti number of the surface $\Sigma$ (i.e. the rank of the first homology group of $\Sigma$).

One way to find upper bounds for $g_4(K)$ and $\gamma_4(K)$ for a knot $K$ is to find a set of band surgeries for $K$ that take $K$ to the unknot (or more generally, a slice knot). A sequence of band surgeries taking $K$ to $K'$ describes a cobordism $W$ between $K$ and $K'$, and if $K'$ bounds a disk in $B^4$, then we can cap off $W$ with that disk to obtain a surface $\Sigma$ in $B^4$ with $\bdy \Sigma=K$. The resulting surface $\Sigma$ is orientable if and only if each of the surgeries are performed on orientation preserving bands. 

A torus knot is a knot which embeds on a standardly embedded torus $T^2\subset S^3$. A pinch band $B$ for a torus knot $K\subset T^2$ is a band attached to $K$ that also embeds on the same torus $T^2$, and a pinch move is surgery on a pinch band. See Figure \ref{pinch}. Since a pinch band embeds on the same torus as $K$, then surgery along a pinch band results in another torus knot. A simple number-theoretic way to compute the resulting torus knot is described in \cite{Jabuka} (included in this paper as Lemma \ref{pinchlemma} for convenience), which can be used to compute the number of pinch moves needed to take a torus knot $K$ to the unknot. We call this the \emph{pinch number} of the torus knot $K$. We note that Lemma \ref{pinchlemma} also implies that the torus knot $K'$ obtained from a pinch move on a torus knot $K$ is unique, and hence there is a unique sequence of pinch moves from any given torus knot to the unknot. 

In a sense, Milnor's conjecture states that the surface realizing $g_4(T(p,q))$ is the most obvious one. Batson's conjecture essentially asks if the same is true for $\gamma_4(T(p,q))$: is $\gamma_4(T(p,q))$ equal to the pinch number of $T(p,q)$? 

%A pinch move on a torus knot is done by performing surgery on a blackboard framed band between adjacent strands on the standard diagram of the torus knot, as shown in figure***. Alternatively, we can define a pinch move without a diagram: if $K$ lives on a torus $T$, then a pinch move is band surgery on any band which lies entirely on $T$ and only intersects $K$ on its boundary.

%
%
%A \emph{rational tangle} is a union of two disjoint arcs $\alpha_1$, $\alpha_2$ embedded in a 3-ball $B$ with $\bdy(\alpha_1\cup\alpha_2)=\bdy(\alpha_1\cup\alpha_2)\cap\bdy B$ equal to the union of four distinguished points $\{x_1, x_2, x_3, x_4\}$. A two bridge knot is the union of two rational tangles, glued along the four distinguished points. 
%
%Let $\tau_1\subset B_1$ and $\tau^2\subset B_2$ be two rational tangles and $\phi:\bdy B_1\to \bdy B_2$ homeomorphism isotopic to the identity . 
%
%Let $S$ be a 2-sphere which decomposes $S^3$ into the union of two 3-ball $B_1$ and $B_2$ with $\bdy B_i=S$, and let $K$ be a knot in $S^3$. We say $K$ is a \emph{2-bridge knot} if we can isotope $K$ in such a way that $K\cap B_i=\alpha_i\cup\beta_i$ is a rational tangle for each $i$. Let $A\in SL(2,\Z)$ be given. Then $A$ defines a homeomorphism 
%
%
%
A rational tangle $(\tau, B)$ is a pair of disjoint arcs $\tau=\alpha\cup\beta$ embedded in a 3-ball $B$ with $\bdy\tau=\bdy\tau\cap\bdy B=\{x_1, x_2, x_3, x_4\}$ for some distinct points $x_i\in\bdy B$ such that there exists an isotopy of $\tau$ rel boundary taking $\tau$ onto $\bdy B$. We can view the torus $T^2$ as a branched double cover of $\bdy B$ branched over the four points $\{x_1, x_2, x_3, x_4\}$. Each rational tangle $(\tau,B)$ can be associated to a fraction $p/q\in\Q\cup \left\{ 1/0 \right\}$ by isotoping $\tau$ onto $\bdy B$ and taking a lift of $\tau$ to a simple closed curve in $T^2$, which is canonically associated to a unique fraction $p/q\in\Q\cup\left\{1/0\right\}$. We sometimes use the notation $\tau\left(p/q \right)$ to denote the rational tangle with fraction $p/q$. 

Assume $\tau\subset \bdy B$ and identify $\bdy B$ with two unit squares $[0,1]\times[0,1]$ glued along their boundary so that $\bdy \tau$ is identified with the four corners, $\alpha=[0,1]\times\{0\}$ lifts to a longitude, and $\beta=\{0\}\times[0,1]$ lifts to a meridian. If $\tau$ cannot be isotoped into one of the unit squares (i.e. $p/q$ is not equal to $0/1$, $1/0$, or $\pm1/1$), then we can compute the associated fraction $p/q$ by counting intersections between $\tau$ with $\alpha$ and $\beta$ (after removing any bigons between $\tau$ and $\alpha$ or $\beta$). If $\tau\cap\alpha=n$ and $\tau\cap\beta=m$, then $p=n+1$ and $q=m+1$.

Any matrix $A\in SL(2,\Z)$ defines a homeomorphism of $T^2$ which takes a $p/q$ curve to an $r/s$ curve, where $A\cdot\begin{bmatrix} p \\ q \end{bmatrix}=\begin{bmatrix} r \\ s \end{bmatrix}$. For any such matrix $A$, there also exists a homeomorphism of $\bdy B\cong S^2$ fixing the four points $\bdy \tau$ setwise (we can think of this as a homeomorphism of the four punctured sphere) which extends to a homeomorphism of $B$ and takes a $p/q$ tangle to an $r/s$ tangle. Let $B_1\cup_S B_2$ be a genus 0 Heegaard splitting of $S^3$. Then $K$ is a 2-bridge knot if and only if $K$ can be isotoped to a knot $K'$ such that $(K'\cap B_i,B_i)$ is a rational tangle for each $i$. So $K'$ can be written as $K'=\tau\left(p_1/q_1\right)\cup\tau\left(p_2/q_2\right)$ for some $p_i/q_i\in\Q\cup\left\{1/0\right\}$. Then a matrix $A\in SL(2,\Z)$ defines a homeomorphism of $B_1\cup B_2$ taking $\tau\left(p_1/q_1\right)\cup\tau\left(p_2/q_2\right)$ to $\tau\left(r_1/s_1\right)\cup\tau\left(r_2/s_2\right)$ where  $A\cdot\begin{bmatrix} p_i \\ q_i \end{bmatrix}=\begin{bmatrix} r_i \\ s_i \end{bmatrix}$ for each $i$. So $K'$ is equivalent to the knot $\tau\left(r_1/s_1\right)\cup\tau\left(r_2/s_2\right)$. (See \cite{MappingClass} for more details.)

\section{The Counterexamples} 

%Uh, define the counterexamples (make pictures and list the two families). Then show how the band shortcuts obtain the slice knots, cite Lisca or whoever that was actually credited to, show how Lobb's example fits in, explain how we can't even restate Batson's conjecture for $T(p,p-1)$. Note that $T(4,9)$ shows up once in both families. 

Our counterexamples are given by the knots $K_n=T(4n, (2n+1)^2)$ and $J_n=T(4n,(2n-1)^2)$. The knot $K_1=T(4,9)$ is Lobb's example, and the knot $K_2=T(8,25)$ is shown in Figure \ref{T825}.

\begin{theorem}
\label{thm:main}
The torus knots $K_n=T(4n,(2n+1)^2)$ (for $n\geq 1$) and $J_n=T(4n, (2n-1)^2)$ (for $n\geq 2$) have pinch number $2n$ and nonorientable 4-ball genus at most $2n-1$. Specifically, there are $2n-1$ band surgeries for $K_n$ and $J_n$ (which are not all orientation preserving) that yield the slice knot with continued fraction expansion $[-(2n+2),-2n]$ and $[-(2n-2),-2n]$ respectively. 

\end{theorem}

Theorem \ref{thm:main} is a direct consequence of Proposition \ref{2} and Proposition \ref{3}, which we will now state and prove. 

%
%    \begin{figure}
%\includegraphics[width=1\textwidth]{T825.png}
%  \caption{The knot $K_2=T(8,25)$ along with three orientation reversing bands. Surgery on these bands yields the slice knot $10_3$ with continued fraction expansion $[4,6]$. Labeling the strands 1-8 from top to bottom, the bands attach strands 1 to 7, 2 to 6, and 3 to 5.}
%  \label{T825}
%\end{figure}
%
%
%

\begin{proposition}
\label{2}

A torus knot of the form $T(4n, (2n\pm 1)^2)$ has pinch number $2n$ (with the exception of $T(4,1)$, which is unknotted). 
\end{proposition}

To prove the proposition, we use the following lemma from \cite{Jabuka}:

\begin{lemma}
\label{pinchlemma}
A pinch move applied to a torus knot $T(p,q)$ yields the torus knot $T(|p-2t|, |q-2h|)$, where $t$ and $h$ are the smallest nonnegative integers satisfying $t\equiv -q^{-1}\mod p$ and $h\equiv p^{-1}\mod q$. 
\end{lemma}

Note that if a pinch move on $T(p,q)$ yields the knot $T(r,s)=T(|p-2t|,|q-2h|)$, then a pinch move on $T(q,p)$ yields the knot $T(s,r)$: Indeed, if $t$ and $h$ are the smallest nonnegative integers satisfying $t\equiv -q^{-1}\mod p$ and $h\equiv p^{-1}\mod q$ and $t'$ and $h'$ are the smallest nonnegative integers satisfying $t'\equiv q^{-1}\mod p$ and $h'\equiv -p^{-1}\mod q$, then $t'=p-t$ and $h'=q-h$. Then $(p-2t)+(p-2t')=(p-2t)+(p-2(p-t))=0$, so $p-2t=-(p-2t')$. Similarly, $q-2h=-(q-2h')$. We get that a pinch move on $T(q,p)$ yields the knot $T(|q-2h'|,|p-2t'|)=T(|q-2h|,|p-2t|)=T(s,r)$. We are now ready to prove Proposition 3.2.

\begin{proof}

We first prove a slightly stronger statement: a pinch move applied to a torus knot $T(p,q)=T((2n\pm1)^2-2k(n\pm1),4n-2k)$ yields the torus knot $T((2n\pm1)^2-2(k+1)(n\pm1),4n-2(k+1))$ (note that we are interchanging the roles of $p$ and $q$ before applying the pinch move, which does not affect the resulting knot). Note that $(2n\pm1)^2-2k(n\pm1)$ can be rewritten as $(4n-2k)(n\pm 1)+1$. 

We first compute $t=-(4n-2k)^{-1} \mod (4n-2k)(n\pm1)+1$. Since $(4n-2k)\cdot (n\pm 1) \equiv -1 \mod (4n-2k)(n\pm1)+1$, then $t=n\pm 1$. Then $$|p-2t|=|(4n-2k)(n\pm 1)+1-2(n\pm1)|=(4n-2(k+1))(n\pm 1)+1.$$

Next we compute $h=((4n-2k)(n\pm 1)+1)^{-1} \mod 4n-2k$. Since $(4n-2k)(n\pm 1)+1\equiv 1 \mod 4n-2k$, then $h=1$. Then $|q-2h|=|(4n-2k)-2|=4n-2(k+1)$. Thus, a pinch move applied to a torus knot $T(p,q)=T((2n\pm1)^2-2k(n\pm1),4n-2k)$ yields the torus knot $T((2n\pm1)^2-2(k+1)(n\pm1),4n-2(k+1))$. 

Now, we repeatedly apply $2n$ pinch moves to the knot $T((2n\pm1)^2,4n)$ to get the knot $T((4n-2(2n))(n\pm1)+1,4n-2(2n))=T(1,0),$ the unknot. Note that if only $2n-1$ pinch moves are applied, we still have a nontrivial knot $$T((4n-2(2n-1))(n\pm1)+1,4n-2(2n-1))=T(2(n\pm1)+1, 2)$$ (with the exception of $T(2(1-1)+1,2)=T(1,2)$). Thus, the pinch number of $T((2n\pm1)^2,4n)$ is $2n$.

\end{proof}
%
%\begin{align*}
%&(4,9)\to (2,5)\to (0,1)\\
%&(8,25)\to (6,19)\to (4,13)\to (2,7)\to (0,1)\\
%&(12,49)\to (10,41)\to (8,33)\to (6,25)\to (4,17)\to (2,9)\to (0,1)\\
%&(16,81)\to (14,71)\to (12,61)\to (10,51)\to (8,41)\to (6,31)\to (4,21)\to (2,11)\to (0,1)\\
%\end{align*}
With the previous lemma in mind, the following tables show a sequence of pinch sequences for $K_n$ and $J_n$ respectively. For clarity, we abbreviate all pinch moves $T(p,q)\to T(r,s)$ by $(p,q)\to (r,s)$.

\begin{table}[h]
\setlength{\tabcolsep}{1pt}
\begin{tabular}{cccccccccccccccccccccccc}
$K_1=(4,9)$&$\to$ &$(2,5)$&$\to$ &$(0,1)$\\
$K_2=(8,25)$&$\to$ &$(6,19)$&$\to$ &$(4,13)$&$\to$ &$(2,7)$&$\to$ &$(0,1$)\\
$K_3=(12,49)$&$\to$ &$(10,41)$&$\to$ &$(8,33)$&$\to$ &$(6,25)$&$\to$ &$(4,17)$&$\to$ &$(2,9)$&$\to$ &$(0,1)$\\
$K_4=(16,81)$&$\to$ &$(14,71)$&$\to$ &$(12,61)$&$\to$ &$(10,51)$&$\to$ &$(8,41)$&$\to$ &$(6,31)$&$\to$ &$(4,21)$&$\to$ &$(2,11)$&$\to$ &$(0,1)$\\
$K_5=(20,121)$&$\to$&$(18,109)$&$\to$&$(16,97)$&$\to$ &$(14,85)$&$\to$ &$(12,73)$&$\to$ &$(10,61)$&$\to$ &$(8,49)$&$\to$ &$\ldots$& &\\
&&&&&&&&$\ldots$& $\to$ &(6,37) &$\to$ & $(4,25)$&$\to$ &$(2,13)$&$\to$ &$(0,1)$\\
\end{tabular}
\caption{\label{tab:KnPinches}Pinch sequences for $K_n=T(4n,(2n+1)^2)$.}
\end{table}

\begin{table}[h]
\setlength{\tabcolsep}{1pt}
\begin{tabular}{cccccccccccccccccccccccc}
$J_2=(8,9)$&$\to$ &$(6,7)$&$\to$ &$(4,5)$&$\to$ &$(2,3)$&$\to$ &$(0,1$)\\
$J_3=(12,25)$&$\to$ &$(10,21)$&$\to$ &$(8,17)$&$\to$ &$(6,13)$&$\to$ &$(4,9)$&$\to$ &$(2,5)$&$\to$ &$(0,1)$\\
$J_4=(16,49)$&$\to$ &$(14,43)$&$\to$ &$(12,37)$&$\to$ &$(10,31)$&$\to$ &$(8,25)$&$\to$ &$(6,19)$&$\to$ &$(4,13)$&$\to$ &$(2,7)$&$\to$ &$(0,1)$\\
$J_5=(20,81)$&$\to$&$(18,73)$&$\to$&$(16,65)$&$\to$ &$(14,57)$&$\to$ &$(12,49)$&$\to$ &$(10,41)$&$\to$ &$(8,33)$&$\to$ &$\ldots$& &\\
&&&&&&&&$\ldots$& $\to$ &(6,25) &$\to$ & $(4,17)$&$\to$ &$(2,9)$&$\to$ &$(0,1)$\\
\end{tabular}
\caption{\label{tab:JnPinches}Pinch sequences for $J_n=T(4n,(2n-1)^2)$.}
\end{table}

The reader may notice from the tables that after four pinch moves to each $J_n$, we obtain the knot $K_{n-2}$ (where $K_0=T(0,1)$). It also seems that the pinch sequence starting at each $K_n$ does not pass through any other $K_m$. Indeed, both of these statements follow as a corollary to (the proof of) Proposition \ref{2}.
%
%$$ (4,9)\to (2,5)\to (0,1)$$
%
%$$(8,25)\to (6,19)\to (4,13)\to (2,7)\to (0,1)$$
%
%$$(12,49)\to (10,41)\to (8,33)\to (6,25)\to (4,17)\to (2,9)\to (0,1)$$
%
%$$(16,81)\to (14,71)\to (12,61)\to (10,51)\to (8,41)\to (6,31)\to (4,21)\to (2,11)\to (0,1)$$
%
%$$(4,9)\to (2,5)\to (0,1)$$
%$$(8,25)\to (6,19)\to (4,13)\to (2,7)\to (0,1)$$
%$$(12,49)\to (10,41)\to (8,33)\to (6,25)\to (4,17)\to (2,9)\to (0,1)$$
%$$(16,81)\to (14,71)\to (12,61)\to (10,51)\to (8,41)\to (6,31)\to (4,21)\to (2,11)\to (0,1)$$
%
%
%$$(4,9)\; \; \ \ \to (2,5)\to (0,1)\qquad\qquad\qquad\qquad\qquad\qquad\qquad\qquad\qquad\qquad\qquad\qquad\qquad\qquad\qquad\qquad\qquad\qquad$$
%$$(8,25)\; \ \to (6,19)\to (4,13)\to (2,7)\to (0,1)\qquad\qquad\qquad\qquad\qquad\qquad\qquad\qquad\qquad\qquad\qquad\qquad$$
%$$(12,49)\to (10,41)\to (8,33)\to (6,25)\to (4,17)\to (2,9)\to (0,1)\qquad\qquad\qquad\qquad\qquad\qquad\qquad\qquad\qquad$$
%$$(16,81)\to (14,71)\to (12,61)\to (10,51)\to (8,41)\to (6,31)\to (4,21)\to (2,11)\to (0,1)\qquad\qquad\qquad\qquad\qquad\qquad$$
%

\begin{corollary}
%The only knots of the form $T(4n,(2n\pm1)^2)$ that have a sequence of pinch moves which passes through $T(4,9)$ are $T(4,9)$ itself and $T(12,25)$.
\label{cor}
Applying a sequence of four pinch moves to $J_n$ yields the knot $K_{n-2}$. 
\end{corollary}

\begin{proof}

It can be seen from the proof of Proposition \ref{2} that applying $4$ pinch moves to the knot $J_n=T(4n,(2n-1)^2)$ yields the knot $T(4n-2\cdot4,(4n-2\cdot 4)(n-1)+1)$. Since $$(4n-2\cdot 4)(n-1)+1=4n^2-12n+9=(2(n-2)+1)^2$$ then a sequence of four pinch moves to $J_n$ yields the knot $T(4(n-2),(2(n-2)+1)^2)=K_{n-2}$.

%The proof of Proposition \ref{2} also shows that all knots in a pinch sequence starting from $K_n=T(4n,(2n+1)^2)$ are of the form $T(4n-2k,(4n-2k)(n+1)+1)$. 
\end{proof}

As noted earlier, Jabuka and Van Cott mentioned that infinitely many counterexamples to Batson's conjecture can be found by working backwards and finding a pinch sequence that passes through $T(4,9)$. In a sense, Corollary \ref{cor} says that all of the counterexamples $J_n$ can be obtained from $K_{n-2}$ by working backwards in the same way, with the exception of $J_2=T(8,9)$ (since $K_0$ is unknotted). The next corollary to Proposition \ref{2} says that the family $\{K_n\}$ consists of knots which cannot be obtained from each other in the same way. 

\begin{corollary}
\label{cor2}
Each nontrivial $K_n$ can not be obtained from any other $K_m$ from a sequence of pinch moves. 
\end{corollary}

\begin{proof}

If $K_n$ is obtained from some other $K_m$ from a sequence of $\ell$ pinch moves, the proof of Proposition \ref{2} also shows that $K_n=T(4n,4n(n+1)+1)$ can be written in the form $$T(4m-2\ell,(4m-2\ell)(m+1)+1).$$  

Then $4m-2\ell=4n$ for some $\ell\in\Z$. This implies $\ell$ is even, so we can instead write $\ell=2k$ and say that $4m-4k=4n$ or $m-k=n$ for some $k\in\Z$. 

%If $K_n$ is obtained from $K_m$ from $\ell=2k$ pinch moves, 

We also have the equation $$(4m-4k)(m+1)+1=4n(n+1)+1.$$ This implies 

\begin{align*}
n(n+k+1)&=n(n+1)\\
n^2+nk+n&=n^2+n\\
nk&=0.
\end{align*}

Since $n\neq 0$, then $k=0$, $\ell=0$, and $n=m$. So each $K_n$ can not be obtained from $K_m$ from a sequence of pinch moves. 

%This implies for any $n$, if $T(4,9)$ is obtained from $T(4n,(2n\pm1)^2)$ by a sequence of pinch moves, then there must be an integer $k$ such that $4n-2k=4$ and $(4n-2k)(n\pm1)+1=9$. The equation $4n-2k=4$ implies $k=2n-2$. We plug this into the equation $(4n-2k)(n\pm1)+1=9$, which has two cases.
%%
%
%For $(4n-2k)(n+1)+1$, we get $4(n+1)+1=9$. Hence $n=1$, $k=2n-2=0$, and $T(4n-2k,(4n-2k)(n+1)+1)=T(4,9)$. For $(4n-2k)(n-1)+1$, we get $4(n-1)+1=9$. Hence $n=3$, $k=2n-2=4$, and $T(4n-2k,(4n-2k)(n-1)+1)=T(12,25)$. 

\end{proof}

We now know that for each $n$, $K_n$ and $J_n$ have pinch number $2n$ (except for the unknotted $J_1$). Next we will show that the nonorientable 4-ball genus of each $K_n$ and $J_n$ is bounded above by $2n-1$ by finding a set of $2n-1$ band surgeries for $K_n$ and $J_n$ that yield the unknot. Any given band is an orientation reversing band for $K_n$ or $J_n$, hence the set of band surgeries describes a nonorientable surface $\Sigma$ bounded by $K_n$ or $J_n$ with $b_1(\Sigma)=2n-1$.

\begin{proposition}
\label{3}

There is a set of $2n-1$ bands for the knot $T(4n, (2n\pm1)^2)$ such that surgery on the set of bands yields the unknot. 
\end{proposition}

\begin{proof}

We consider $T(4n,(2n\pm1)^2)$ as the closure of a braid on $4n$ strands. Label the strands from 1 to $4n$ reading left to right. The braid has $m=n\pm1$ full twists, as well as strand $4n$ crossing over strands $4n-1$ through 1. We attach blackboard framed bands from strand $k$ to $4n-k$ for each $k=1, \ldots 2n-1$. The situation is depicted in Figure \ref{braidwithbands}. After performing surgery on the bands, we can isotope the braid to have two sets of $2n-1$ semicircles opposite each other. We can pull the bottom half of the semicircles through the $m$ full twists to obtain the diagram in Figure \ref{afterbands}.

\begin{figure}[h]
\centering
\begin{minipage}[b]{.5\textwidth}
  \centering
  \includegraphics[width=.4\linewidth]{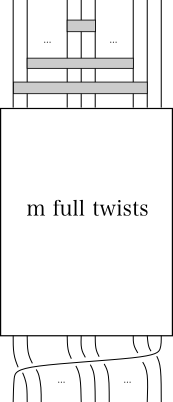}
  \captionof{figure}{A braid whose closure is $T(4n, (2n\pm1)^2)$, along with $2n-1$ bands.}
  \label{braidwithbands}
\end{minipage}%
\begin{minipage}[b]{.5\textwidth}
  \centering
  \includegraphics[width=.4\linewidth]{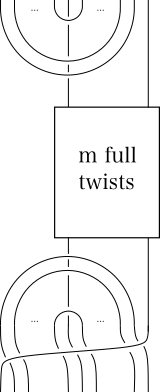}
  \captionof{figure}{We obtain this figure after performing band surgery and an isotopy.}
  \label{afterbands}
\end{minipage}
\end{figure}

    \begin{figure}
\includegraphics[width=.3\textwidth]{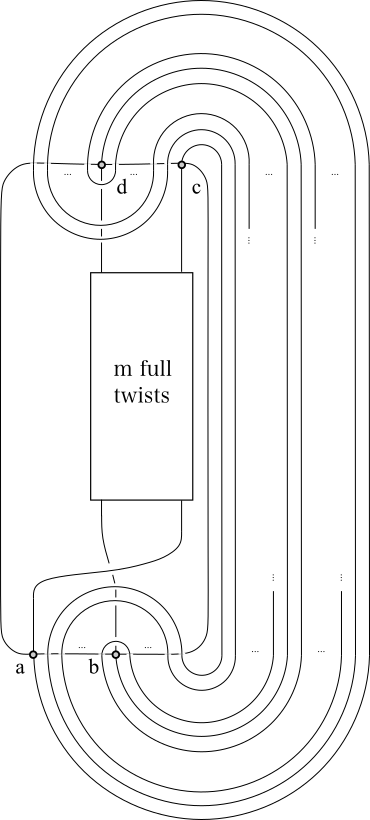}
  \caption{The knot $K_2=T(8,25)$ along with three orientation reversing bands. Surgery on these bands yields the slice knot $10_3$ with continued fraction expansion $[4,6]$. Labeling the strands 1-8 from top to bottom, the bands attach strands 1 to 7, 2 to 6, and 3 to 5.}
  \label{closure}
\end{figure}

After pulling the bottom half of the semicircles underneath the last strand (and looking at the braid closure of our previous diagram), we obtain the knot diagram in Figure \ref{closure}. Note that this is a 2-bridge knot, as there is a 2-sphere $S$ (also shown in Figure \ref{closure}) which intersects the knot in the four points labeled $a, b, c$, and $d$ and splits the knot into two trivial tangles.  Recall that there is a correspondence between 2-strand trivial tangles and rational numbers (as described above). The first trivial tangle consists of $m$ full twists and an additional half twist, i.e. $2m+1$ half twists. We can compute the slope of this first tangle to be $\frac{1}{2m+1}$, since we can isotope it onto a unit square $[0,1]\times[0,1]$ such that the tangle intersects the square in its four corners along with $2m$ intersection points on the left and right sides of the square.

For the second tangle, we first isotope the tangle by pulling the points labeled $c$ and $d$ clockwise around the knot to obtain the picture in Figure \ref{outsidebox1}. Now, if we draw arcs horizontally from $a$ to $b$, $b$ to $c$, $c$ to $d$, and $d$ to $a$ (moving right, passing the point at infinity, then appearing from the left side of the page), we can count intersection points of these arcs with the tangle. In fact, the number of intersection points in Figure \ref{outsidebox1} between the tangle and any of the four arcs is equal to $2n-1$, exactly the number of bands we performed surgery on. We can isotope the tangle to remove one intersection point between the tangle and the arcs from $b$ to $c$ and from $d$ to $a$, shown in Figure \ref{outsidebox2}. Since the tangle intersects the arcs from $a$ to $b$ and $c$ to $d$ in $2n-1$ points and the arcs from $b$ to $c$ and $d$ to $a$ in $2n-2$ points, the slope of this tangle is $\frac{2n}{2n-1}$ (see \cite{Zieschang} for more details and exposition on this construction of a rational tangle).

We now have that our 2-bridge knot is the union of tangles whose rational slopes are $\frac{1}{2m+1}$ (where $m=n\pm1$) and $\frac{2n}{2n-1}$. That is, we can write $K$ as $\tau\hspace{-1mm}\left(\frac{1}{2m+1}\right)\cup\tau\hspace{-1mm}\left(\frac{2n}{2n-1}\right)$.
% It is known that a union of two rational tangles with slopes $\frac{p_1}{q_1}$ and $\frac{r_1}{s_1}$ is isotopic to a union of rational tangles with slopes $\frac{p_2}{q_2}$ and $\frac{r_2}{s_2}$ if there is a matrix $A\in SL(2,\Z)$ such that $A\cdot \begin{bmatrix} p_1 \\ q_1 \end{bmatrix}=\begin{bmatrix} p_2 \\ q_2 \end{bmatrix}$ and $A\cdot \begin{bmatrix} r_1 \\ s_1 \end{bmatrix}=\begin{bmatrix} r_2 \\ s_2 \end{bmatrix}$ (CITATION OR PROOF?). 
 Consider the matrix $A=\begin{bmatrix} 1 & 0 \\ -2m-1 & 1 \end{bmatrix}\in SL(2,\Z)$.

\begin{figure}
\centering
\begin{minipage}[b]{.5\textwidth}
  \centering
  \includegraphics[width=1\linewidth]{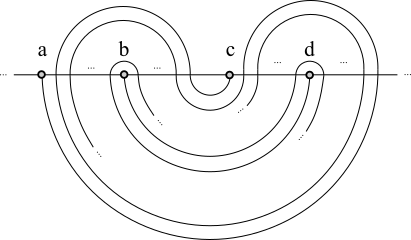}
  \captionof{figure}{The knot in Figure \ref{closure} is a union of two tangles. We isotope the second tangle by pulling the points labeled $c$ and $d$ counterclockwise around the knot to obtain this diagram.}
  \label{outsidebox1}
\end{minipage}%
\begin{minipage}[b]{.5\textwidth}
  \centering
  \includegraphics[width=1\linewidth]{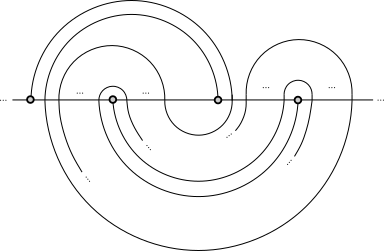}
  \captionof{figure}{After an isotopy, the tangle intersects the horizontal arcs from $a$ to $b$ and $c$ to $d$ in $2n-1$ points and the arcs from $b$ to $c$ and $d$ to $a$ in $2n-2$ points.}
  \label{outsidebox2}
\end{minipage}
\end{figure}

Observe that $\begin{bmatrix} 1 & 0 \\ -2m-1 & 1 \end{bmatrix} \begin{bmatrix} 1\\ 2m+1 \end{bmatrix}=\begin{bmatrix} 1 \\ 0 \end{bmatrix} $ and 
\begin{align*}
\begin{bmatrix} 1 & 0 \\ -2m-1 & 1 \end{bmatrix} \begin{bmatrix} 2n\\ 2n-1 \end{bmatrix}&=\begin{bmatrix} 1 & 0 \\ -2(n\pm1)-1 & 1 \end{bmatrix} \begin{bmatrix} 2n\\ 2n-1 \end{bmatrix} \\
&=\begin{bmatrix} 2n \\ -4n(n\pm1)-1 \end{bmatrix}.
\end{align*}
Then, from the discussion in section 2 (see also \cite{MappingClass}), $K$ is equivalent to the knot $\tau\hspace{-1mm}\left(\frac{1}{0}\right)\cup\tau\hspace{-1mm}\left(\frac{2n}{-4n(n\pm 1)-1}\right)$, also sometimes referred to as the denominator closure of the rational tangle $\tau\hspace{-1mm}\left(\frac{2n}{-4n(n\pm1)-1}\right)$ (see \cite{Kauffman}). We can compute the continued fraction expansion as $$\frac{2n}{-4n(n\pm1)-1}=\frac{1}{\frac{-4n(n\pm1)-1}{2n}}=\frac{1}{-(2n\pm 2)+\frac{1}{-2n}}.$$ Siebenman states in \cite{Siebenman} that Casson, Gordon, and Conway showed all knots of this form are slice (see also \cite{Lisca} for a classification of slice 2-bridge knots). For convenience, see Figure \ref{sliceband} to see a slice band for the knot with continued fraction expansion $[n+2,n]$.

%\begin{figure}
%\centering
%\begin{minipage}[b]{.5\textwidth}
%  \centering
%  \includegraphics[width=.5\linewidth]{sliceband.png}
%  \captionof{figure}{A slice band for the knot with continued fraction expansion $[n+2,n]$.}
%  \label{sliceband}
%\end{minipage}%
%\begin{minipage}[b]{.5\textwidth}
%  \centering
%  \includegraphics[width=.5\linewidth]{aftersliceband.png}
%  \captionof{figure}{After the band surgery, we obtain the two component unlink.}
%  \label{aftersliceband}
%\end{minipage}
%\end{figure}

\begin{figure}
\centering
\begin{subfigure}{.5\textwidth}
  \centering
  \includegraphics[width=.5\linewidth]{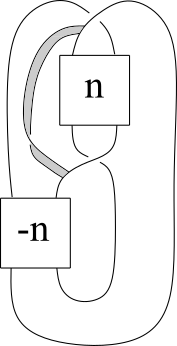}
%  \caption{A slice band for the knot with continued fraction expansion $[n+2,n]$.}
  \label{a}
\end{subfigure}%
\begin{subfigure}{.5\textwidth}
  \centering
  \includegraphics[width=.5\linewidth]{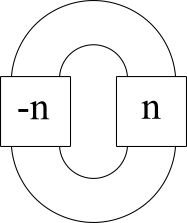}
%  \caption{After the band surgery, we obtain the two component unlink.}
  \label{b}
\end{subfigure}
\caption{A slice band for the knot with continued fraction expansion $[n+2,n]$. After the band surgery, we obtain the two component unlink.}
\label{sliceband}
\end{figure}

\end{proof}

%
%$T(p,q)=T((2n\pm1)^2-2k(n\pm1),4n-2k)$ yields the torus knot $T((2n\pm1)^2-2(k+1)(n\pm1),4n-2(k+1))$ (note that we are interchanging the roles of $p$ and $q$ before applying the pinch move, which does not affect the resulting knot). Note that $(2n\pm1)^2-2k(n\pm1)$ can be rewritten as $(4n-2k)(n\pm 1)+1$. 
%
%We first compute $t=-(4n-2k)^{-1} \mod (4n-2k)(n\pm1)+1$. Since $(4n-2k)\cdot (n\pm 1) \equiv -1 \mod (4n-2k)(n\pm1)+1$, then $t=n\pm 1$. Then $$|p-2t|=|(4n-2k)(n\pm 1)+1-2(n\pm1)|=(4n-2(k+1))(n\pm 1)+1.$$
%
%Next we compute $h=((4n-2k)(n\pm 1)+1)^{-1} \mod 4n-2k$. Since $(4n-2k)(n\pm 1)+1\equiv 1 \mod 4n-2k$, then $h=1$. Then $|q-2h|=|(4n-2k)-2|=4n-2(k+1)$. Thus, a pinch move applied to a torus knot $T(p,q)=T((2n\pm1)^2-2k(n\pm1),4n-2k)$ yields the torus knot $T((2n\pm1)^2-2(k+1)(n\pm1),4n-2(k+1))$. 

  \section{Questions}

%Maybe put the significance of the counterexample discussion here instead. Discuss how the lower bounds for $\gamma_4$ are not that helpful because they're still lower than the new upper bounds I provided for $\gamma_4$, so there is still work to be done to actually compute $\gamma_4$ for $T(p,q)$. Are there torus knots with $\gamma_4$ at most the pinch number minus two? 

As noted earlier, Milnor's conjecture asserts that the most efficient way to obtain $g_4$ for a torus knot is also the nicest or most obvious way. The pinch bands which motivated Batson's conjecture were in a sense the nicest band surgeries one can find to obtain the unknot. However, the bands found for our families of counterexamples also have a nice symmetric property to them. Are there counterexamples that do not have a certain amount of symmetry to them? Or is this the best that we can do? We ask the following.

\begin{question}
\label{q1}
Are there other examples of torus knots with nonorientable 4-ball genus less than its pinch number?
\end{question}
While we expect the answer to this question to be yes, it would be significant if the answer were no. Presently, all known counterexamples to Batson's conjecture have nonorientable 4-ball genus only one less than the pinch number, hence showing the answer to Question \ref{q1} is no would prove that Batson's conjecture was only off by one - that is, all torus knots have nonorientable 4-ball genus equal to either its pinch number or its pinch number minus one. Should the answer to Question \ref{q1} be yes, it is natural to wonder how large the difference between the pinch number and nonorientable 4-ball genus can be.

%how much smaller the nonorientable 4-ball genus can be than the pinch number.

\begin{question}
Can the difference between the pinch number and the nonorientable 4-ball genus of a torus knot be arbitrarily large?
\end{question}

Batson verified his conjecture for $T(2k,2k-1)$ using a lower bound for $\gamma_4$. Osv\'ath, Stipsicz and Szab\'o provided another lower bound $\nu-\frac{1}{2}\sigma$ for $\gamma_4$ in \cite{OSS}, and Jabuka and Van Cott gave a combinatorial way to compute this lower bound in certain cases in \cite{Jabuka}. In particular, Jabuka and Van Cott in \cite{Jabuka} gave a categorization of when $\nu-\frac{1}{2}\sigma$ is equal to the pinch number minus one for a torus knot. The \emph{sign} of a pinch move $T(p,q)\to T(r,s)$ is the sign of the value $p-2t$ or $q-2h$, where $t$ and $h$ are the smallest nonnegative integers such that $t\equiv -q^{-1}\mod p$ and $h\equiv p^{-1}\mod q$. Let $p,q>1$ be relatively prime with $q$ odd, and let $T(p,q)=T(p_n,q_n)\to T(p_{n-1},q_{n-1})\to\ldots \to T(p_0,q_0)=U$ (with $q_0=1$) be a sequence of pinch moves from $T(p,q)$ to the unknot $U$. Jabuka and Van Cott proved that if $p$ is even, then $\nu(T(p,q))-\frac{1}{2}\sigma(T(p,q))=n-1$ if and only if there is exactly one index $k$ such that the sign of the pinch move $T(p_k,q_k)\to T(p_{k-1},q_{k-1})$ is negative (Proposition 5.1(b) of \cite{Jabuka}). From the proof of Lemma \ref{pinchlemma}, we see that the sign of every pinch move $T((2n\pm1)^2-2k(n\pm1),4n-2k)\to T((2n\pm1)^2-2(k+1)(n\pm1),4n-2(k+1))$ is equal to the sign of $p-2t=4n-2(k+1))(n\pm 1)+1$, which is positive (except for when $p-2t=0$, in which case $q-2h$ can still be seen to be positive). Following the discussion prior to Lemma \ref{pinchlemma}, reversing the roles of $p$ and $q$ (since to apply Jabuka and Van Cott's result, we require $q$ to be odd), we see that the sign of every pinch move $T(4n-2k,(2n\pm1)^2-2k(n\pm1))\to T(4n-2(k+1),(2n\pm1)^2-2(k+1)(n\pm1))$ is negative. Thus, $(\nu-\frac{1}{2}\sigma)(K_n)< 2n-1$ for each $n>0$, and $(\nu-\frac{1}{2}\sigma)(J_n)< 2n-1$ for each $n>1$ (the exception $n=1$ occurs since $J_1$ is already unknotted). It would be interesting to see if other lower bounds can be used to compute the exact value of $\gamma_4$ for $K_n$ and $J_n$. We hence ask the following. 
\begin{question}
Is the nonorientable 4-ball genus of $K_n$ and $J_n$ actually equal to $2n-1$?
\end{question}

\bibliographystyle{amsalpha}
\bibliography{vince}

 \end{document}